\documentclass[english]{amsart}
\usepackage{amsfonts, amsmath, amssymb, amscd, amsthm, array, graphicx}
\usepackage[latin9]{inputenc}
\usepackage{amsmath}
\usepackage{stmaryrd}
\usepackage{graphicx}
\usepackage{amssymb}
\usepackage{dsfont}
\usepackage{babel}
\usepackage{color}
\usepackage[all]{xy}

\usepackage{mathtools}
\usepackage{hyperref,url}

\makeatletter
\def\blfootnote{\xdef\@thefnmark{}\@footnotetext}
\makeatother
\usepackage{tikz}
\usetikzlibrary{arrows,shapes,snakes,automata,backgrounds,petri}
\usetikzlibrary{automata} 
\usetikzlibrary[automata] 
\usetikzlibrary{matrix,decorations.pathreplacing}
\usepackage{geometry}
\usepackage{pdflscape}
\usepackage{graphicx}
\usetikzlibrary{external,automata,trees,positioning,shadows,arrows,shapes.geometric}

\newlength\Textht
\newtheorem{thm}{Theorem}[section]
\newtheorem{cor}[thm]{Corollary}
\newtheorem{lem}[thm]{Lemma}
\newtheorem{prop}[thm]{Proposition}
\newtheorem{que}[thm]{Question}

\theoremstyle{definition}

\newtheorem{obs}[thm]{Observation}

\newtheorem{rem}[thm]{Remark}

\newcommand{\R}[1]{{\color{red} #1}}
\newcommand{\B}[1]{{\color{blue} #1}}

\newcommand{\rk}{\operatorname{r}}
\newcommand{\id}{\operatorname{Id}}

\newcommand{\Z}{{\mathbb Z}}
\newcommand\ZZ{\mathbb{Z}}
\newcommand\FF{\mathbb{F}}
\newcommand{\HH}{\mathcal{H}}
\newcommand{\GG}{\mathcal{G}}
\newcommand{\KK}{\mathcal{K}}

\newcommand{\Mat}{\operatorname{M}}
\newcommand{\GL}{\operatorname{GL}}
\newcommand{\Aut}{\operatorname{Aut}}

\newcommand{\End}{\operatorname{End}}

\newcommand{\im}{\operatorname{Im}}
\newcommand{\Fix}{\operatorname{Fix}}
\newcommand{\acl}{\operatorname{a-Cl}}
\newcommand{\ecl}{\operatorname{e-Cl}}
\newcommand{\astab}{\operatorname{a-Stab}}
\newcommand{\stab}{\operatorname{e-Stab}}

\newcommand{\leqfg}{\leqslant_{\rm f.g.}}
\newcommand{\leqfi}{\leqslant_{\rm f.i.}}

\begin{document}

\title{Computation of endo-fixed closures in free-abelian times free groups}

\author{Mallika Roy}
\address{Departament of Mathematics, Universidad del Pa\'{\i}s Vasco (UPV/EHU), Spain.} \email{mallika.roy@upc.edu}

\author{Enric Ventura}
\address{Departament de Matem\`atiques, Universitat Polit\`ecnica de Catalunya, CATALONIA.} \email{enric.ventura@upc.edu}



\begin{abstract}
In this paper, we explore the behaviour of the fixed subgroups of endomorphisms of free-abelian times free (FATF) groups. We exhibit an algorithm which, given a finitely generated subgroup $\mathcal{H}$ of a FATF group $\mathcal{G}$, decides whether $\mathcal{H}$ is the fixed subgroup of some (finite) family of endomorphisms of $\GG$ and, in the affirmative case, it finds such a family. The algorithm combines both combinatorial and algebraic methods.
\end{abstract}

\maketitle

\section{Introduction}\label{intro}

The problem of studying fixed subgroups of automorhisms and endomorphisms of certain families of groups has a long history in the  literature of the last decades. The case of free groups $\FF$ is particularly interesting and reach. The story began with the paper Dyer--Scoot~\cite{DS}, where the authors proved that finite order automorphisms $\varphi\in \Aut(\FF_n)$ of the rank $n<\infty$ free group $\FF_n$ have its fixed subgroup $\Fix\varphi =\{w\in \FF_n \mid w\varphi=w\}$ always being a free factor of $\FF_n$ itself. This is not true in general, but soon P. Scoot conjectured that $r(\Fix\varphi)\leq n$ for every automorphism $\varphi\in \Aut(\FF_n)$. This conjecture motivated an intense research activity during the period of the 1980's, proving it in several special cases, and culminating in the paper Bestvina--Handel~\cite{BH}, where the conjecture was proved in its full generality. Apart from the result itself, this paper has been very influential as train-tracks for free group automorphisms were first introduced and studied in~\cite{BH}. The same rank inequality works even for general endomorphisms $\varphi\in \End(\FF_n)$; see Imrich--Turner~\cite{IT}.

Such a deep result, far from exhausting this line of research, stimulated further investigations focused on other properties of fixed subgroups. A relevant one was the computability of a free basis for $\Fix \varphi$, when $\varphi\in \Aut(\FF_A)$ is given by the images of the generators $A$ (with the notation $\FF_A$ we mean the free group on the alphabet $A$; so, $\FF_A \simeq \FF_{|A|}$). In the finite rank case, $|A|<\infty$, and for automorphisms, this was proved in Bogopolski--Maslakova~\cite{BM}, and alternatively in Feighn--Handel~\cite{FH}, both making strong use of the train-track technology, and providing quite complicated and high complexity algorithms. Moreover, the initial versions of the train-track theory were working only for automorphisms (or, at most, for injective endomorphisms), and several analogous questions for general endomorphisms remained intractable. In particular, the computability of (free basis for) fixed subgroups remained an open problem for endomorphisms of finitely generated free groups, until the very recent general solution provided by Mutanguha~\cite{M}: 

\begin{thm}[Bogopolski--Maslakova~\cite{BM}, Feighn--Handel~\cite{FH}, Mutanguha~\cite{M}]\label{computefix} There is an algorithm which, given an endomorphism $\varphi\in \End(\FF_n)$ by images of generators, it computes a free basis for $\Fix \varphi$. 
\end{thm}

See the survey Ventura~\cite{V2} for many more details of this interesting story.

In a group $G$, a subgroup $H\leqslant G$ is said to be \emph{auto-fixed} (resp., \emph{endo-fixed}) if it is of the form $H=\Fix S$ for some $S\subseteq \Aut(G)$ (resp., some $S\subseteq \End(G)$). It is said that $H$ is \emph{1-auto-fixed} (resp., \emph{1-endo-fixed}) when the set $S$ can be taken to be of cardinal one. In the free case, $G=\FF_n$, for $n=2$ all these families do coincide but, for $n\geq 3$, it is known that the family of 1-auto-fixed subgroups is strictly contained in that of 1-endo-fixed subgroups; see Martino--Ventura~\cite{MV3}. Interestingly, for $n\geq 3$, it is not known whether the families of 1-auto-fixed (resp., 1-endo-fixed) and auto-fixed (resp., endo-fixed) subgroups do coincide or not; to the best of our knowledge, the most known in this direction is that auto-fixed (resp., endo-fixed) subgroups of $\FF_n$ are free factors of 1-auto-fixed (resp., 1-endo-fixed) subgroups of $\FF_n$, while free factors of 1-auto-fixed subgroups are not, in general, even endo-fixed; see Martino--Ventura~\cite{MV}. 

In the paper Ventura~\cite{V}, a kind of dual point of view was considered: instead of starting with a morphism $\varphi$ and asking about $\Fix \varphi$, one can start with a given subgroup $H\leqslant G$ and ask about the (pointwise) \emph{auto-stabilizer} and \emph{endo-stabilizer} of $H$, i.e., $\astab_{G}(H)=\{\varphi\in \Aut(G) \mid H\leqslant \Fix \varphi\}$ and $\stab_{G}(H)=\{\varphi\in \End(G) \mid H\leqslant \Fix \varphi\}$, respectively, a subgroup of $\Aut(G)$ and a submonoid of $\End(G)$. And more interestingly, following the terminology from~\cite{V}, one can consider the so-called \emph{auto-fixed closure} (resp., \emph{endo-fixed closure}) of $H$ with respect to $G$, as the smallest subgroups fixed by all automorphisms (resp., endomorphisms) of $G$ fixing $H$, namely, 
 $$
\acl_{G}(H)=\Fix \big( \astab_{G}(H) \big) =\bigcap_{\tiny \begin{array}{c} \varphi\in \Aut(G) \\ H\leqslant \Fix \varphi \end{array}} \Fix\varphi,
 $$
and 
 $$
\ecl_{G}(H)=\Fix \big( \stab_{G}(H) \big) =\bigcap_{\tiny \begin{array}{c} \varphi\in \End(G) \\ H\leqslant \Fix \varphi \end{array}} \Fix\varphi
 $$
(where, for any set of endomorphisms $S$, we write $\Fix S=\bigcap_{\varphi\in S} \Fix \varphi$). Observe that, in general, $\astab_{G}(H)\leqslant \stab_{G}(H)$ and so, $\ecl_{G}(H) \leqslant \acl_{G}(H)$, sometimes with strict inequality, for example, because of the mentioned result from Martino--Ventura~\cite{MV3} about the existence of 1-endo-fixed subgroups of $\FF_n$, $n\geq 3$, which are not auto-fixed. Note also that a subgroup $H\leqslant G$ is auto-fixed (resp., endo-fixed) if and only if $\acl_{G}(H)=H$ (resp., $\ecl_{G}(H)=H$). 

For example it is well known that, for every $1\neq w\in \FF_n$ and $0\neq r\in \ZZ$, the equation $X^r=w^r$ has a \emph{unique} solution in $\FF_n$, which is the obvious one $X=w$; this means that any endomorphism $\varphi\colon \FF_n \to \FF_n$ fixing $w^r$ \emph{must} also fix $w$. As a consequence, a cyclic subgroup $H\leqslant \FF_n$ is auto-fixed, if and only if it is endo-fixed, and if and only if it is maximal as a cyclic subgroup. For later use, we note that the same is true in free abelian groups $\ZZ^m$, and in direct products of the form $\ZZ^m\times \FF_n$.

In the context of free groups, auto-stabilizers are quite well behaved while endo-stabilizers could be intricate. A classical result from McCool~\cite{McCool} (see also Lyndon--Schupp~\cite[Prop.~I.5.7]{LS}) already established that, for a finitely generated subgroup $H\leqfg \FF_n$, the auto-stabilizer $\astab_{\FF_n}(H)$ is again finitely generated and computable, i.e., one can compute explicit automorphisms $\varphi_1, \ldots ,\varphi_k\in \Aut(\FF_n)$ such that $\astab_{\FF_n}(H)=\langle \varphi_1, \ldots ,\varphi_k \rangle \leqslant \Aut(\FF_n)$ (even more: $\astab_{\FF_n}(H)$ is finitely presented, and a full finite set of relations among the $\varphi_i$'s can also be computed). Quite contrasting, the case of endomorphisms is much more delicate: in the paper Ciobanu--Dicks~\cite{CD} the authors gave an explicit example of a finitely generated subgroup $H\leqfg \FF_n$ such that its endo-stabilizer $\stab_{\FF_n}(H)$ is \emph{not finitely generated} as a submonoid of $\End(\FF_n)$ (with the obvious consequent complications for studying algorithmic issues).

Focusing again on algorithmic problems, in the paper Ventura~\cite{V} it was proved that, if $H\leqfg \FF_n$ then both $\acl_{\FF_n}(H)$ and $\ecl_{\FF_n}(H)$ are also finitely generated and computable:

\begin{thm}[{Ventura, \cite[Thm. 9, 11]{V}}]\label{thm for free endo}
Let $H\leqfg \FF_n$ be given by a finite set of generators. Then, the auto-fixed closure $\acl_{\FF_n}(H)$ (resp., the endo-fixed closure $\ecl_{\FF_n}(H)$) of $H$ is also finitely generated, and a free-basis for it is algorithmically computable, together with a set of $k\leq 2n$ automorphisms $\varphi_1, \ldots, \varphi_k \in \Aut(\FF_n)$ (resp., endomorphisms $\varphi_1, \ldots, \varphi_k \in \End(\FF_n)$), such that $\acl_{\FF_n}(H)= \Fix \varphi_1 \cap \cdots \cap \Fix \varphi_k$ (resp., $\ecl_{\FF_n}(H)= \Fix \varphi_1 \cap \cdots \cap \Fix \varphi_k$).
\end{thm}

With respect to this result, it is interesting to focus on the big difference between the situation for automorphisms and for endomorphisms. For automorphisms, this is a straightforward consequence from previous results: given $H\leqfg \FF_n$, one can use McCool's Theorem and compute automorphisms $\varphi_1, \ldots ,\varphi_k\in \Aut(\FF_n)$ such that $\astab_{\FF_n}(H)=\langle \varphi_1, \ldots ,\varphi_k \rangle \leqslant \Aut(\FF_n)$; then, use Bogopolski--Maslakova's Theorem to compute free bases for the subgroups $\Fix\varphi_1, \ldots ,\Fix\varphi_k$; and, finally, use the pull-back technique (see~\cite{KM}) to compute a free basis for the finite intersection $\Fix\varphi_1 \cap \cdots \cap \Fix\varphi_k =\Fix(\astab_{\FF_n}(H))=\acl_{\FF_n}(H)$ (a more detailed analysis allows to take $k\leq 2n$). The situation for endomorphisms, although apparently similar, is much harder: on one hand the endo-stabilizer $\stab_{\FF_n}(H)$ need not be finitely generated and, on the other hand, the fixed subgroup of an endomorphism $\varphi\in \End(\FF_n)$ was not known to be computable at the time Theorem~\ref{thm for free endo} was written. Even though these two obstacles ruined completely the possibility to adapt the previous argument to the endomorphism case, the author of~\cite{V} managed to prove the above result for endomorphisms with a completely different strategy: he constructed finitely many intermediate subgroups $H\leqfg K_1,\ldots ,K_t\leqfg \FF_n$ such that $\ecl_{\FF_n}(H)=\acl_{K_1}(H)\cap \cdots \cap \acl_{K_t}(H)$. 

Of course, as an immediate corollary, one gets the algorithmic recognition of auto-fixed (resp., endo-fixed) subgroups:

\begin{cor}[{Ventura, \cite[Cor. 10, 12]{V}}]\label{Enric-deciding}
It is algorithmically decidable whether a given $H\leqfg \FF_n$ is auto-fixed (resp., endo-fixed) and, in case it is, compute a set of $k\leq 2n$ automorphisms $\varphi_1, \ldots, \varphi_k \in \Aut(\FF_n)$ (resp., endomorphisms $\varphi_1, \ldots, \varphi_k \in \End(\FF_n)$) such that $H=\Fix \varphi_1 \cap \cdots \cap \Fix \varphi_k$.
\end{cor}

\bigskip

Let us move now into the bigger and more complicated family of free-abelian times free groups, namely groups of the form $\GG=\ZZ^m\times \FF_n$. In this more general setting an extra difficulty shows up quickly: as soon as $m\geqslant 1$ and $n\geqslant 2$, fixed subgroups of automorphisms of $\GG$ need not be finitely generated any more. The simplest example of this new behaviour is the automorphism $\Psi$ of $\ZZ\times \FF_2=\langle t\rangle\times \langle a,b\rangle$ mapping $t\mapsto t$, $a\mapsto ta$, $b\mapsto b$; since it maps $t^{r}w(a,b)$ to $t^{r+|w|_a}w(a,b)$, it is easy to see that its fixed subgroup 
 $$
\Fix \Psi =\{ t^{r}w(a,b) \mid |w|_a=0\}=\langle t \rangle \times \ll b\gg_{\langle a,b\rangle}
 $$
is normal and not finitely generated. Apart from this, the group $\GG=\ZZ^m\times \FF_n$, as soon as $m\geq 1$ and $n\geq 2$, is known to be not Howson, i.e., it contains finitely generated subgroups $\HH,\KK\leqfg \GG$ whose intersection $\HH\cap \KK$ is not finitely generated (consider, for example, $\HH=\langle a, b\rangle$, $\KK=\langle ta, b\rangle$, and $\HH\cap \KK =\ll b\gg_{\langle a,b\rangle}$). Of course, these two facts make the treatment of fixed subgroups and fixed closures much more involved, specially for the arguments of algorithmic nature. The first results from the literature in this direction are the following: 

\begin{thm}[{Delgado--Ventura, \cite[Cor~6.7]{DV}}]\label{DV-fix}
Let $\GG=\ZZ^m\times \FF_n$. There is an algorithm which, given an automorphism $\Psi\in \Aut(\GG)$ (by the images of generators), decides whether $\Fix\Psi$ is finitely generated and, if so, computes a basis for it. The same is true for endomorphisms $\Psi\in \End(\GG)$, assuming a free basis for the fixed subgroup of an endomorphism of $\FF_n$ is computable. 
\end{thm}

See Section~\ref{basics} for the concept of basis in the context of free-abelian times free groups. Note that, at present time, the last assumption in Theorem~\ref{DV-fix} is not a restriction anymore because of the result of Mutanguha~\cite{M}; see Theorem~\ref{computefix}.

\begin{thm}[{Roy--Ventura, \cite[Prop.~5.13]{RV}}]\label{RV-fix}
Let $\GG=\ZZ^m\times \FF_n$. There is an algorithm which, given automorphisms $\Psi_1, \ldots ,\Psi_{\ell}\in \Aut(\GG)$ (by the images of generators), decides whether $\Fix\Psi_1\cap \cdots \cap \Fix\Psi_{\ell}$ is finitely generated and, if so, computes a basis for it.
\end{thm}

Note that, in order to show Theorem~\ref{RV-fix}, it is not enough to apply Theorem~\ref{DV-fix} finitely many times and then somehow control the intersection of the computed fixed subgroups: it could very well happen that some of the $\Fix \Psi_i$'s are not finitely generated while the full intersection $\Fix\Psi_1\cap \cdots \cap \Fix\Psi_{\ell}$ is. Or it could also happen that, even with all of the individual $\Fix \Psi_i$'s and the full intersection $\Fix \Psi_1 \cap \cdots \cap \Fix \Psi_{\ell}$ being finitely generated, \emph{all} of the intermediate intersections are not finitely generated; see Delgado--Roy--Ventura~\cite{DRV}. The proof of Theorem~\ref{RV-fix} uses a completely different strategy. 

\begin{rem}\label{remark}
We observe that Theorem~\ref{RV-fix} extends to endomorphisms with exactly the same proof given in Roy--Ventura~\cite[Prop.~5.13]{RV} except that, at the point when one needs to compute fixed subgroups of certain endomorphisms of $\FF_n$, one needs to invoke Mutanguha's Theorem~\ref{computefix} instead of Bogopolski--Maslakova's one (this was not reported in~\cite{RV} because Mutanguha's result was an open problem at that time).  
\end{rem}

With the help of Theorem~\ref{RV-fix}, M. Roy and E. Ventura managed to compute auto-closures of finitely generated subgroups in free-abelian times free groups (together with the previous decision on whether they are finitely generated), and proved the analog to Theorem~\ref{thm for free endo} and Corollary~\ref{Enric-deciding} for automorphisms: 

\begin{thm}[Roy--Ventura, \cite{RV}]\label{closure}
Let $\GG=\ZZ^m\times \FF_n$. There is an algorithm which, given a finite set of generators for a subgroup $\HH\leqfg \GG$, outputs a finite set of automorphisms $\Psi_1, \ldots, \Psi_{\ell} \in \Aut(\GG)$ such that $\acl_{G}(\HH)=\Fix \Psi_1 \cap \cdots \cap \Fix \Psi_{\ell}$, decides whether this is finitely generated and, in case it is, computes a basis for it. 
\end{thm}

\begin{cor}[Roy--Ventura, \cite{RV}]\label{deciding}
One can algorithmically decide whether a given $\HH\leqfg \GG$ is auto-fixed and, in case it is, compute a finite set of automorphisms $\Psi_1, \ldots, \Psi_{\ell} \in \Aut(\GG)$ such that $\HH=\Fix \Psi_1 \cap \cdots \cap \Fix \Psi_{\ell}$.
\end{cor}

Note that this leaves the possibility of a finitely generated subgroup $\HH\leqfg \GG$ having a non finitely generated auto-closure $\acl_{\GG}(\HH)$; however, as far as we are aware of, no such example is known, and the implication $\HH\leqfg \GG$ $\Rightarrow$ $\acl_{\GG}(\HH)\leqfg \GG$ remains an open problem, as stated in~\cite{RV}.

Sorting out the possibly non-finitely generated nature of the main involved characters, we are able to extend Theorem~\ref{closure} and Corollary~\ref{deciding} to endomorphisms of free-abelian times free groups. The main results in the present paper are the following, answering Question~5.9(ii) from Roy--Ventura~\cite{RV}:

\begin{thm}\label{main}
Let $\GG=\ZZ^m\times \FF_n$. There is an algorithm which, given a finite set of generators for a subgroup $\HH\leqfg \GG$, outputs a finite set of endomorphisms $\mathcal{S}_0 =\{\Psi_1, \ldots, \Psi_{\ell}\} \in \End(\GG)$ such that $\ecl_{\GG}(\HH)=\Fix \Psi_1 \cap \cdots \cap \Fix \Psi_{\ell}$, decides whether this is finitely generated and, in case it is, computes a basis for it. 
\end{thm}

\begin{cor}\label{maincor}
One can algorithmically decide whether a given $\HH\leqfg \GG$ is endo-fixed and, in case it is, compute a finite set of endomorphisms $\Psi_1, \ldots, \Psi_{\ell} \in \End(\GG)$ such that $\HH=\Fix \Psi_1 \cap \cdots \cap \Fix \Psi_{\ell}$.
\end{cor}

The extension from Theorem~\ref{closure} and Corollary~\ref{deciding} to Theorem~\ref{main} and Corollary~\ref{maincor} has to face another technical complication: the existence of the more degenerate type-(II) endomorphims (non of which is an automorphism). The paper is organized as follows. In Section~\ref{intro} we have explained the context of this line of research. In Section~\ref{basics} we introduce the general notation for free-abelian times free groups and remind the explicit form of the two types of endomorphims of such groups. In Section~\ref{abelian} we analyze the special case where $\HH$ is abelian, which is when type-(II) endomorphims have a significant role. In Section~\ref{non-abelian} we consider the case $\HH$ not abelian (the main one) where type-(II) endomorphisms play no role. Finally, in Section~\ref{open} we collect a list of questions and open problems. 

We advice the reader that the provided algorithms are theoretical and far from effective, in the sense that their complexities are quite high. A natural question is whether there exist more natural and efficient algorithms, say polynomial time, for solving such problems.

\subsection*{General notation and conventions}

For a group $G$, $\Aut(G)$ (resp., $\End(G)$) denotes the group (resp., the monoid) of automorphisms (resp., endomorphisms) of $G$. We write them all with the argument on the left, that is, we denote by $(x)\varphi$ (or simply $x\varphi$) the image of the element $x$ by the homomorphism $\varphi$; accordingly, we denote by  $\varphi \psi$ the composition $A \xrightarrow{\varphi} B \xrightarrow{\psi} C$. We will denote by $\Mat_{n\times m}(\ZZ)$ the $n\times m$ (additive) group of matrices over $\ZZ$, and by $\GL_m(\ZZ)$ the linear group over the integers. Following the same convention above, when thinking a matrix $A$ as a map, it will always act on the right of horizontal vectors, $\mathbf{v}\mapsto \mathbf{v}A$. 

To easily distinguish among them, we shall use regular letters $u,v,w,\ldots$ to refer to elements from the free group $\FF_n$, boldfaced letters $\mathbf{u}, \mathbf{v}, \mathbf{a}, \mathbf{b}, \ldots$ to refer to vectors from $\ZZ^m$, capital letters to refer to integral matrices $P\in \Mat_{n\times m}(\ZZ)$, $Q\in \Mat_{m\times m}(\ZZ)$ (and $I$ to denote the identity matrix of any dimension), lowercase greek letters $\varphi, \psi, \gamma, \sigma, \ldots$ to refer to homomorphisms (and $\id$ to refer the identity morphism) of the free group $\FF_n$, and uppercase greek letters $\Psi, \Phi, \ldots$ to refer to homomorphisms of free-abelian times free groups $\GG=\ZZ^m\times \FF_n$. Finally, we shall use usual capital letters $H,K,\ldots$ to denote subgroups of $\FF_n$, and calligraphic capitals $\HH,\KK,\ldots$ to denote subgroups of $\GG$. In particular, we shall need to work with the abelianization map $\rho\colon \FF_n \twoheadrightarrow \ZZ^n$ (w.r.t. a free ambient basis always clear from the context), for which we shall use the notation $u\mapsto \mathbf{u}:=u\rho \in \ZZ^n$.

Throughout the paper we write $H \leqslant G$ (resp. $H\leqfg G$ and $H\leqfi G$) to express that $H$ is a subgroup (resp. a finitely generated subgroup, and a finite index subgroup) of $G$, reserving the symbol $\leq$ for inequalities among real numbers.

\section{Basics on free-abelian times free groups}\label{basics}

Any direct product of a free-abelian group, $\ZZ^m$, $m\geq 0$, and a free group, $\FF_n$, $n\geq 0$, will be called, for short, a \emph{free-abelian times free} group, $\GG=\Z^m\times \FF_n$. We will work in $\GG$ with multiplicative notation (as it is a non-abelian group as soon as $n\geq 2$) but want to refer to its subgroup $\ZZ^m\leqslant \GG$ with the standard additive notation (elements thought as row vectors with addition). To make these compatible, consider the standard presentations $\ZZ^m=\langle t_1, \ldots ,t_m \mid [t_i,t_j],\,\,\, i,j=1,\ldots ,m\rangle$ and $\FF_n=\langle z_1, \ldots ,z_n \mid \,\,\rangle$, and the standard normal form for elements from $\GG$ with vectors on the left, namely $t_1^{a_1}\cdots t_m^{a_m}w(z_1, \ldots ,z_n)$, where $a_1,\ldots ,a_m\in \Z$ and $w\in \FF_n$ is a reduced word on the alphabet $Z=\{z_1, \ldots ,z_n\}$; then, let us abbreviate this in the form
 $$
t_1^{a_1}\cdots t_m^{a_m}w(z_1, \ldots ,z_n) =t^{(a_1, \ldots ,a_m)}w(z_1,\ldots ,z_n)=t^{\mathbf{a}} w(z_1,\ldots ,z_n),
 $$
where $\mathbf{a}=(a_1,\ldots ,a_m)\in \ZZ^m$ is the row vector made with the integers $a_i$'s, and $t$ is a meta-symbol serving only as a pillar for holding the vector $\mathbf{a}=(a_1,\ldots ,a_m)$ up in the exponent. This way, the operation in $\GG$ is given by $(t^{\mathbf{a}} u)(t^{\mathbf{b}} v)=t^{\mathbf{a}} t^{\mathbf{b}} uv=t^{\mathbf{a+b}} uv$ in multiplicative notation, while the abelian part works additively, as usual, up in the exponent. We denote by $\pi$ the natural projection to the free part, $\pi\colon \ZZ^m\times \FF_n \twoheadrightarrow \FF_n$, $t^{\mathbf{a}} u\mapsto u$, whose kernel is $\ZZ^m=\{t^{\mathbf{a}} \mid \mathbf{a}\in \ZZ^m\}\unlhd \GG$, the \emph{purely abelian} part of $\GG$.

According to Delgado--Ventura~\cite[Def.~1.3]{DV}, a \emph{basis} of a subgroup $\HH\leqslant \GG$ is a set of generators for $\HH$ of the form $\{ t^{\mathbf{a_1}}u_1, \ldots ,t^{\mathbf{a_r}}u_r; t^{\mathbf{b_1}},\ldots ,t^{\mathbf{b_s}} \}$, where $\mathbf{a_1}, \ldots ,\mathbf{a_r}\in \ZZ^m$, $\{ u_1, \ldots ,u_r\}$ is a free-basis of $\HH\pi\leqslant \FF_n$, and $\{ \mathbf{b_1},\ldots ,\mathbf{b_s}\}$ is an abelian-basis of $L_{\HH} =\HH\cap \Z^m \leqslant \ZZ^m$ (so, $s\leq m$); this includes the possibility that $\HH\pi\leqslant \FF_n$ is not finitely generated, in which case $r=\infty$ and the above basis is made of infinitely many elements. To avoid confusions, we reserve the word \emph{basis} in the above sense for $\GG$, in contrast with \emph{abelian-basis} and \emph{free-basis} for the corresponding concepts in $\ZZ^m$ and $\FF_n$, respectively. It was showed in~\cite[Prop.~1.9]{DV} that every subgroup $\HH\leqslant \GG$ admits a basis, algorithmically computable from any given finite set of generators in the finitely generated case. As a consequence, any subgroup $\HH\leqslant \ZZ^m\times \FF_n$ is again free-abelian times free, $\HH\simeq \ZZ^{m'}\times \FF_{n'}$, for some $0\leq m'\leq m$ and some $0\leq n'\leq \infty$; in particular, $\HH$ is finitely generated if and only if $\HH\pi\leqslant \FF_n$ is so.

We recall from Delgado--Ventura~\cite[Props.~5.1,~5.2(iii)]{DV} that every endomorphism $\Psi$ of the group $\GG=\ZZ^m \times \FF_n$, $n\geq 2$, is either of the form 
\begin{itemize}
\item[(I)] $\Psi =\Psi_{\varphi,Q,P}\colon \GG\to \GG$, $t^{\mathbf{a}}u\mapsto t^{\mathbf{a}Q+\mathbf{u}P}(u\varphi)$, where $\varphi \in \End(\FF_n)$, $Q\in \Mat_{m\times m}(\ZZ)$, $P\in \Mat_{n\times m}(\ZZ)$, and $\rho\colon \FF_n \to \ZZ^n$, $u\mapsto u\rho=\mathbf{u}$ is the abelianization map;
or
\item[(II)] $\Psi_{z,\boldsymbol{\ell} ,\mathbf{h},Q,P} \colon \GG \to \GG$, $t^{\mathbf{a}}u \mapsto t^{\mathbf{a}Q+\mathbf{u}P} z^{\mathbf{a}\boldsymbol{\ell}^T+\mathbf{u}\mathbf{h}^T}$, where $1\neq z\in \FF_n$ is not a proper power, $Q\in \Mat_{m\times m}(\ZZ)$, $P\in M_{n\times m}(\ZZ)$, $\mathbf{0}\neq \boldsymbol{\ell} \in \ZZ^m$, and $\mathbf{h}\in \ZZ^n$ ($\boldsymbol{\ell}^T, \mathbf{h}^T$ denote the vectors $\boldsymbol{\ell}, \mathbf{h}$ written as columns); in this case, $z$ is called the \emph{stable letter} of $\Psi_{z,\boldsymbol{\ell} ,\mathbf{h},Q,P}$.
\end{itemize}
It is easy to see that none of the type-(II) endomorphisms is actually an automorphism, and that a type-(I) $\Psi_{\varphi,Q,P}$ is an automorphism if and only if $\varphi\in \Aut(\FF_n)$ and $Q\in \GL_m(\ZZ)$. 

Note that a subgroup $\HH\leqslant \GG$ is abelian if and only if $\HH\pi$ is cyclic: in case $\HH\pi=1$, we have $\HH\leqslant \ZZ^m\leqslant \GG$; otherwise, $\HH\pi=\langle z^r\rangle$ for some $1\neq z\in \FF_n$ not being a proper power and some $0\neq r\in \ZZ$, and $\HH$ is then of the form $\HH=\langle t^{\mathbf{a}}z^r, t^{\mathbf{b_1}},\ldots ,t^{\mathbf{b_s}}\rangle$, where $\mathbf{a}\in\ZZ^m$ and $\{\mathbf{b_1}, \ldots ,\mathbf{b_s}\}$ is an abelian-basis of $L_{\HH}=\HH\cap \ZZ^m$. The following facts are straightforward. 

\begin{obs}\label{stable letter obs}
(i) Let $\HH=\langle t^{\mathbf{a}}z^r, t^{\mathbf{b_1}},\ldots ,t^{\mathbf{b_s}}\rangle$ be an abelian subgroup of $\GG$ not contained in $\ZZ^m$ (i.e., with $1\neq z\in \FF_n$ not being a proper power). Then any type-(II) endomorphism $\Psi$ fixing $\HH$ has stable letter $z^{\pm 1}$.

(ii) The image, and so the fixed subgroup, of any type-(II) endomorphism are abelian subgroups of $\GG$. 

(iii) If $\HH\leqslant \GG$ is not abelian then $\stab_{\GG}(\HH)$ contains no type-(II) endomorphism. 
\end{obs}

\section{Computing endo-closures in FATF groups: the abelian case}\label{abelian}

The rest of the paper is dedicated to prove Theorem~\ref{main}. Suppose we are given a finite set of generators for a finitely generated subgroup $\HH\leqfg \GG$ (with a few extra calculations, we can always assume that this set of generators is actually a finite basis for $\HH$). Let us analyze $\ecl_{\GG}(\HH)$, i.e., the intersection of the fixed subgroups of all endomorphisms of $\GG$ fixing $\HH$. Of course, among those endomorphisms there could be of both types (I) and (II). However, according to Observation~\ref{stable letter obs}(iii), when $\HH$ is not abelian no endomorphism in $\stab_{\GG}(\HH)$ is of type-(II). Hence, let us first analyze separately the case $\HH$ abelian and then, in the next section, we can forget completely about the more technical type-(II) endomorphisms. 

Let $\HH\leqslant \GG$ be an abelian subgroup.

If $\HH\leqslant \ZZ^m$ then $L_{\HH}=\HH$ and it is clear that the endo-closure of $\HH$ in $\GG$ is $\tilde{L}_{\HH}$, the smallest direct summand of $\ZZ^m$ containing $L_{\HH}$: $L_{\HH}\leqfi \tilde{L}_{\HH}\leqslant \tilde{L}_{\HH}\oplus C=\ZZ^m$. In fact, every vector in $\tilde{L}_{\HH}$ has a multiple in $L_{\HH}$ so, $\tilde{L}_{\HH}\leqslant \ecl_{\GG}(\HH)$. Conversely, $\tilde{L}_{\HH}$ is the fixed subgroup of the type-(I) automorphism $\Psi_{\sigma, Q, P}$, where $\sigma\colon \FF_n \to \FF_n$, $z_i\mapsto z_i^{-1}$ inverts every letter (so, $\Fix \sigma =\{1\}$), $Q$ fixes $\tilde{L}_{\HH}$ and changes signum over $C$ (so, $\Fix Q=\tilde{L}_{\HH}$), and $P=0$. Therefore, $\ecl_{\GG}(\HH)=\tilde{L}_{\HH}$. Using basic linear algebra techniques we can easily compute an abelian-basis for $\tilde{L}_{\HH}$ from the given (abelian-) basis for $L_{\HH}=\HH$.  

Suppose then that $\HH$ is an abelian subgroup of $\GG$ not contained in $\ZZ^m$. That is, $\HH=\langle t^{\mathbf{a}}z^r, t^{\mathbf{b_1}},\ldots ,t^{\mathbf{b_s}}\rangle\leqslant \GG$ with $1\neq z\in \FF_n$ not being a proper power, $r\neq 0$, and $\{\mathbf{b_1}, \ldots ,\mathbf{b_s}\}$ being an abelian-basis of $B=\langle \mathbf{b_1}, \ldots ,\mathbf{b_s}\rangle \leqslant \ZZ^m$, with $L_{\HH}=\HH\cap \ZZ^m=\langle t^{B}\rangle$. In this case, the following proposition gives us explicitly the endo-closure of $\HH$; to state it, we need some previous notation. 

Consider $B\leqfi \tilde{B}\leqslant \tilde{B}\oplus C=\ZZ^m$ as before, and split the vector $\mathbf{a} \in \ZZ^m$ as $\mathbf{a}=\mathbf{b}+\mathbf{c}$, with $\mathbf{b}\in \tilde{B}$ and $\mathbf{c}\in C$. If $\mathbf{a}\not\in \tilde{B}$ then $\mathbf{c}\neq \mathbf{0}$ and we let $\alpha$ be the greatest common divisor of the coordinates of $\mathbf{c}$; we have $\mathbf{c}=\alpha \mathbf{c'}$ with $\mathbf{c'}$ having relatively prime coordinates. Also, let $\alpha_1 =\gcd (\alpha, r)$; thus, we have $\alpha=\alpha_1 \alpha_2$ and $r=\alpha_1r'$, with $\gcd(\alpha_2,r')=1$. Let $\mathbf{z}\in \ZZ^n$ be the abelianization of $z\in \FF_n$ (note that $\mathbf{z}$ may be $\mathbf{0}$ even when $z\neq 1$).

\begin{prop}
Let $\HH=\langle t^{\mathbf{a}}z^r, t^{\mathbf{b_1}},\ldots ,t^{\mathbf{b_s}}\rangle\leqslant \GG$ be an abelian subgroup of $\GG$, with $1\neq z\in \FF_n$ not being a proper power, $r\neq 0$, and $\{\mathbf{b_1}, \ldots ,\mathbf{b_s}\}$ being an abelian-basis of $B$, with $L_{\HH}=\HH\cap \ZZ^m=\langle t^B\rangle$. With the above notation, we have: 
\begin{itemize}
\item[(i)] if $\mathbf{a}\in \tilde{B}$ then $\ecl_{\GG}(\HH)=\langle z,\, t^{\tilde{B}}\rangle$;
\item[(ii)] if $\mathbf{a}\not\in \tilde{B}$ and $\gcd (\mathbf{z}, \alpha_2)=1$ then $\ecl_{\GG}(\HH)=\langle t^{\alpha_2 \mathbf{c'}}z^{r'},\, t^{\tilde{B}}\rangle$;
\item[(iii)] if $\mathbf{a}\not\in \tilde{B}$ and $\gcd (\mathbf{z}, \alpha_2)\neq 1$ then $\ecl_{\GG}(\HH)=\langle z,\, t^{\tilde{B}},\, t^{\mathbf{c'}} \rangle$. 
\end{itemize}
Moreover, $\ecl_{\GG}(\HH)$ is always 1-endo-fixed and an endomorphism $\Psi\in \End(\GG)$ satisfying $\Fix\Psi=\ecl_{\GG}(\HH)$ is computable. 
\end{prop}

\begin{proof}
For (i), suppose $\mathbf{a}\in \tilde{B}$ (and so, $\mathbf{c}=\mathbf{0}$). In this case, $\ecl_{\GG}(\HH)$ must contain $t^{B}$ and so, $t^{\tilde{B}}$; therefore, it also contains $z^r$ and so, $z$. Hence, $\langle z,\, t^{\tilde{B}}\rangle\leqslant \ecl_{\GG}(\HH)$. For the other inclusion, consider the type-(I) automorphism $\Psi_{\gamma_z, Q, P}$, where $\gamma_z \colon \FF_n \to \FF_n$, $x\mapsto z^{-1}xz$ is the right conjugation by $z$ (so, $\Fix \gamma_z =\langle z\rangle$, since $z$ is not a proper power), $Q\colon \ZZ^m\to \ZZ^m$ fixes $\tilde{B}$ and changes signum over $C$ (so, $\Fix Q=\tilde{B}$), and $P=0$. We have
 $$
\Fix \Psi_{\gamma_z, Q, P} =\{t^{\mathbf{a}}u \mid t^{\mathbf{a}Q}(u\gamma_z)=t^{\mathbf{a}}u\} =\langle z,\, t^{\tilde{B}}\rangle
 $$
and hence, $\ecl_{\GG}(\HH)=\langle z,\, t^{\tilde{B}}\rangle$, as we wanted to show.

Assume then that $\mathbf{a}\not\in \tilde{B}$ (and so, $\mathbf{c}\neq \mathbf{0}$). In this case, $\ecl_{\GG}(\HH)$ must contain $t^{B}$ and so, $t^{\tilde{B}}$; therefore, it also contains the element 
 $$
t^{c}z^{r}=t^{\alpha_1 \alpha_2 \mathbf{c'}}z^{\alpha_1 r'}=(t^{\alpha_2 \mathbf{c'}}z^{r'})^{\alpha_1}
 $$
and so, $t^{\alpha_2 \mathbf{c'}}z^{r'}$. Hence, $\langle t^{\alpha_2 \mathbf{c'}}z^{r'},\, t^{\tilde{B}}\rangle \leqslant \ecl_{\GG}(\HH)$. Under the extra assumption $\gcd (\mathbf{z}, \alpha_2)=1$, we are going to see that this is in fact an equality, proving (ii); otherwise, the endo-closure of $\HH$ will be a bit bigger. 

So, suppose (ii), $\gcd (\mathbf{z}, \alpha_2)=1$, and compute B\'ezout coefficients $\mathbf{h}\in \ZZ^n$ and $\rho\in \ZZ$ satisfying $\mathbf{z}\mathbf{h}^T +\alpha_2\rho =1$. Now, construct an abelian basis of $\ZZ^m$ of the form $\{ \mathbf{\tilde{b}_1},\ldots ,\mathbf{\tilde{b}_s}, \mathbf{c'}, \mathbf{c_1},\ldots ,\mathbf{c_{m-s-1}}\}$, where $\{ \mathbf{\tilde{b}_1},\ldots \mathbf{\tilde{b}_s}\}$ is an abelian basis for $\tilde{B}$ (this is possible because $\tilde{B}\oplus C=\ZZ^m$ and $\mathbf{c'}\in C$ has relatively prime coordinates); and consider the automorphism defined by $Q\colon \ZZ^m\to \ZZ^m$, $\mathbf{\tilde{b}_i}\mapsto \mathbf{\tilde{b}_i}$, $\mathbf{c'}\mapsto \mathbf{c'}$, $\mathbf{c_j}\mapsto -\mathbf{c_j}$, for $i=1,\ldots ,s$ and $j=1,\ldots ,m-s-1$. It is clear that $\Fix Q=\tilde{B}\oplus \langle \mathbf{c'}\rangle$. Construct also a vector $0\neq \boldsymbol{\ell} \in \ZZ^m$ satisfying $\mathbf{b}\boldsymbol{\ell}^T=0$ for every $\mathbf{b}\in \tilde{B}$, and $\mathbf{c'}\boldsymbol{\ell}^T=r'\rho$ (again, this is possible because $\tilde{B}\oplus C=\ZZ^m$ and $\mathbf{c'}\in C$ has relatively prime coordinates). With all these ingredients, consider the type-(II) endomorphism $\Psi_{z,\boldsymbol{\ell} ,\mathbf{h},Q,0} \colon \GG \to \GG$, $t^{\mathbf{a}}u \mapsto t^{\mathbf{a}Q} z^{\mathbf{a}\boldsymbol{\ell}^T+\mathbf{u}\mathbf{h}^T}$. We have that 
 $$
\Fix \Psi_{z,\boldsymbol{\ell} ,\mathbf{h},Q,0}=\{ t^{\mathbf{a}}u \mid \mathbf{a}Q=\mathbf{a}, \phantom{a} u=z^{\mathbf{a}\boldsymbol{\ell}^T+\mathbf{u}\mathbf{h}^T}\}.
 $$
Note that the first condition tells us that $\mathbf{a}=\mathbf{b}+\lambda \mathbf{c'}$, for some $\mathbf{b}\in \tilde{B}$ and some $\lambda\in \ZZ$; and, abelianizing the second one, we get $\mathbf{u}=(\mathbf{a}\boldsymbol{\ell}^T+\mathbf{u}\mathbf{h}^T)\mathbf{z}$. Therefore, 
 $$
\mathbf{a}\boldsymbol{\ell}^T+\mathbf{u}\mathbf{h}^T=(\mathbf{b}+ \lambda \mathbf{c'})\boldsymbol{\ell}^T+ (\mathbf{a} \boldsymbol{\ell}^T+ \mathbf{u}\mathbf{h}^T)\mathbf{z}\mathbf{h}^T = \lambda r'\rho + (\mathbf{a} \boldsymbol{\ell}^T+ \mathbf{u}\mathbf{h}^T)(1-\alpha_2 \rho )
 $$
and so, $\alpha_2 (\mathbf{a} \boldsymbol{\ell}^T +\mathbf{u}\mathbf{h}^T)=\lambda r'$. Since $\alpha_2$ and $r'$ are relatively prime, $\lambda$ must be multiple of $\alpha_2$, i.e., $\lambda\in \alpha_2 \ZZ$. From this we deduce that 
 $$
\begin{array}{cl}
\Fix \Psi_{z,\boldsymbol{\ell} ,\mathbf{h},Q,0} & =\{t^{\mathbf{b}+\lambda \mathbf{c'}} z^{\frac{\lambda}{\alpha_2} r'} \mid \mathbf{b}\in \tilde{B},\,\, \lambda \in \alpha_2\ZZ \}\\
& = \langle t^{\tilde{B}},\,\, t^{\lambda \mathbf{c'}} z^{\frac{\lambda}{\alpha_2} r'} \mid \lambda \in \alpha_2\ZZ \rangle \\ & = \langle t^{\alpha_2 \mathbf{c'}} z^{r'},\,\, t^{\tilde{B}}\rangle.
\end{array}
 $$
Thus, $\ecl_{\GG}(\HH) =\langle t^{\alpha_2 \mathbf{c'}} z^{r'},\,\, t^{\tilde{B}}\rangle$, as we wanted to see.

Finally, suppose (iii), i.e., $\gcd (\mathbf{z}, \alpha_2)\neq 1$. In this situation, we claim that $\stab_{\GG}(\HH)$ contains only type-(I) endomorphisms. In fact, suppose that a type-(II) endomorphism $\Psi$ fixes $\HH$, and let us find a contradiction. By Observation~\ref{stable letter obs}(i), its stable letter mut be $z^{\epsilon}$, $\epsilon=\pm 1$, and so, $\Psi=\Psi_{z^{\epsilon},\boldsymbol{\ell} ,\mathbf{h},Q,P}$, for some integral vectors $\mathbf{0}\neq \boldsymbol{\ell}\in \ZZ^m$ and $\mathbf{h}\in \ZZ^n$, and some integral matrices $Q\in \Mat_{m\times m}(\ZZ)$ and $P\in \Mat_{n\times m}(\ZZ)$. By the argument above, $\Psi_{z,\boldsymbol{\ell} ,\mathbf{h},Q,P}$ must also fix the element 
 $$
t^{\alpha_2 \mathbf{c'}} z^{r'} \mapsto t^{\alpha_2 \mathbf{c'}Q+r'\mathbf{z}P} z^{\epsilon (\alpha_2 \mathbf{c'}\boldsymbol{\ell}^T+ r'\mathbf{z}\mathbf{h}^T)};
 $$
in particular, $\alpha_2\mathbf{c'}\boldsymbol{\ell}^T + r'\mathbf{z}\mathbf{h}^T= \epsilon r'$ and $r' \mid \alpha_2\mathbf{c'}\boldsymbol{\ell}^T$. But $\gcd(\alpha_2, r')=1$ so, $r' \mid \mathbf{c'}\boldsymbol{\ell}^T$, i.e., $\mathbf{c'}\boldsymbol{\ell}^T=\mu r'$ for some $\mu\in \ZZ$. Now, simplifying $r'$ from the last equation, we have $\alpha_2\mu+\mathbf{z}\mathbf{h}^T=\epsilon$, which is a contradiction with the assumption $\gcd(\mathbf{z}, \alpha_2)\neq 1$. This proves the claim. 

Finally, knowing that $\ecl_{\GG}(\HH)$ contains no type-(II) endomorphism, we can improve our current inclusion $\langle t^{\alpha_2 \mathbf{c'}}z^{r'},\,\, t^{\tilde{B}} \rangle \leqslant \ecl_{\GG}(\HH)$: any endomorphism fixing $\HH$ is of type-(I) and also fixes $t^{\alpha_2 \mathbf{c'}}z^{r'}$ hence, it must fix $z^{r'}$ and so, $z$ as well; therefore, it must further fix $t^{\alpha_2 \mathbf{c'}}$ and hence, $t^{\mathbf{c'}}$. Summarizing, $\langle z,\, t^{\tilde{B}},\, t^{\mathbf{c'}} \rangle \leqslant \ecl_{\GG}(\HH)$. For the other inclusion, consider the type-(I) automorphism $\Psi_{\gamma_z, Q, P}$, where $\gamma_z \colon \FF_n \to \FF_n$, $x\mapsto z^{-1}xz$ is the right conjugation by $z$, $Q\colon \ZZ^m\to \ZZ^m$ is the automorphism from case (ii), and $P=0$; it is straightforward to see that $\Fix \Psi_{\gamma_z, Q, P}=\langle z,\, t^{\tilde{B}},\, t^{\mathbf{c'}} \rangle$. Therefore, $\ecl_{\GG}(\HH)=\langle z,\, t^{\tilde{B}},\, t^{\mathbf{c'}} \rangle$, as we wanted to prove.

To conclude the proof, observe that in the three cases $\ecl_{\GG}(\HH)$ has been obtained as the fixed subgroup of a single endomorphism $\Psi\in \End(\GG)$ explicitly determined by $\HH$. 
\end{proof}

This proves (a bit more than) the special case of Theorem~\ref{main} for abelian subgroups $\HH$. 

\begin{cor}
Let $\GG=\ZZ^m\times \FF_n$. If $\HH\leqslant \GG$ is abelian then $\ecl_{\GG}(\HH)$ is finitely generated and 1-endo-fixed. Moreover, from a given finite set of generators for $\HH$, one can algorithmically compute a basis for $\ecl_{\GG}(\HH)$ and an endomorphism $\Psi\in \End(\GG)$ such that $\ecl_{\GG}(\HH)=\Fix \Psi$. \qed
\end{cor}

\section{Computing endo-closures in FATF groups: the non-abelian case}\label{non-abelian}

In this section we concentrate on the case $\HH$ not abelian. As argued above, this implies that $\stab_{\GG}(\HH)$ contains only type-(I) endomorphism so, we can ignore those of type-(II). Recall that a type-(I) endomorphism is of the form $\Psi=\Psi_{\varphi,Q,P}\colon \GG\to \GG$, $t^{\mathbf{a}}u\mapsto t^{\mathbf{a}Q+\mathbf{u}P}(u\varphi)$, where $\varphi \in \End(\FF_n)$, $Q\in \Mat_{m\times m}(\ZZ)$, and $P\in \Mat_{n\times m}(\ZZ)$. For notational convenience, we write $\Psi\pi=\varphi$. The following are a couple of initial observations about type-(I) endomorphisms.  

\begin{lem}\label{prop A}
Let $\HH\leqfg \GG$ be non-abelian. Then, $(\stab_{\mathcal{G}}(\mathcal{H}))\pi =\stab_{\FF_n}(\mathcal{H}\pi)$.
\end{lem}

\begin{proof}
Let $\varphi \in (\stab_{\mathcal{G}}(\mathcal{H}))\pi$. This means that there exists $Q\in \Mat_{m\times m}(\Z)$ and $P\in \Mat_{n\times m}(\Z)$ such that $\Psi_{\varphi, Q, P} \in \stab_{\mathcal{G}}(\mathcal{H})$. Hence $\mathcal{H} \leqslant \Fix {\Psi_{\varphi, Q, P}}$ and so, $\mathcal{H}\pi \leqslant \Fix \varphi$. Thus, $\varphi \in \stab_{\FF_n}(\mathcal{H}\pi)$.

To prove the other inclusion, let $\varphi \in \stab_{\FF_n}(\mathcal{H}\pi)$. Take $Q=I$ and $P=0$, and let us consider the type-(I) endomorphism $\Psi_{\varphi, I, 0}\in \End(\GG)$. For every $t^{\mathbf{a}}u\in \HH$, we have $u\in \mathcal{H}\pi$ and so, $u\varphi=u$ and $(t^{\mathbf{a}}u )\Psi_{\varphi, I, 0}=t^{\mathbf{a}}u$. Therefore, $\mathcal{H} \leqslant \Fix \Psi_{\varphi, I, 0}$, i.e., $\Psi_{\varphi, I, 0}\in \stab_{\mathcal{G}}(\mathcal{H})$. We deduce that $\varphi =(\Psi_{\varphi, I, 0})\pi \in (\stab_{\mathcal{G}}(\mathcal{H}))\pi$.
\end{proof}

\begin{lem}
Let $\HH\leqfg \GG$ be non-abelian. Then, $(\ecl_{\mathcal{G}}(\mathcal{H}))\pi \leqslant \ecl_{\FF_n}(\mathcal{H}\pi)$, with the equality being not true, in general. 
\end{lem}

\begin{proof}
Let $u\in (\ecl_{\mathcal{G}}(\mathcal{H}))\pi$. Then, there exists a vector $\mathbf{a}\in \Z^m$ such that $t^{\mathbf{a}}u \in \ecl_{\mathcal{G}}(\mathcal{H})$ and so $(t^{\mathbf{a}}u)\Psi=t^{\mathbf{a}}u$ for all $\Psi\in \stab_{\mathcal{G}}(\mathcal{H})$. Let us take now an arbitrary $\varphi \in \stab_{\FF_n}(\mathcal{H}\pi)$. In Lemma~\ref{prop A} we proved that $\Psi_{\varphi, I, 0} \in \stab_{\mathcal{G}}(\mathcal{H})$. Therefore, $t^{\mathbf{a}}u \in \Fix \Psi_{\varphi, I, 0}$ and $u \in \Fix \varphi$. We deduce that $u\in \ecl_{\FF_n}(\mathcal{H}\pi)$. This completes the proof of the inclusion. 

As a counterexample for the other inclusion, consider $\GG=\ZZ\times \FF_2=\langle t\rangle \times \langle a,b\rangle$ and the non-abelian subgroup $\HH=\langle ta^2,\, b\rangle$. It is clear that $\HH\pi=\langle a^2, b\rangle$ and $\ecl_{\FF_2}(\HH\pi)=\langle a, b\rangle$. However, $a\not\in (\ecl_{\GG}(\HH))\pi$: in fact, if that was the case, there should exist an integer $\textbf{x}\in \ZZ$ such that $t^{\textbf{x}}a\in \ecl_{\GG}(\HH)$; but the endomorphism $\Psi_{\id, Q, P}\in \End(\GG)$ with $Q=(-1)$ and $P=\binom{1}{0}$, fixes both $ta^2\mapsto t^{1\cdot (-1)+(2,0)P}a^2=ta^2$ and $b\mapsto t^{\textbf{0}\cdot (-1)+\textbf{(0,1)}P}b=b$ so it would also fix $t^{\mathbf{x}}a\mapsto t^{\textbf{x}\cdot (-1)+\textbf{(1,0)}P}a=t^{\textbf{-x+1}}a$, meaning $2\textbf{x}=1$, which is a contradiction.
\end{proof}

We are ready to start the proof of the main result in the non-abelian case.  

\begin{prop}\label{prop main}
Let $\HH \leqfg \GG$ be non-abelian (given by a a finite set of generators). Then, there is a computable family of endomorphisms $S_0=\{ \id =\varphi_0, \varphi_1, \ldots ,\varphi_k\}\subseteq \stab_{\FF_n}(\HH\pi)$, $k\leq 2n$, such that
 $$
\ecl_{\mathcal{G}}(\mathcal{H})= \bigcap_{{\tiny \begin{array}{c} \mathcal{H} \leqslant \Fix {\Psi_{\varphi, Q, P}} \end{array}}} \Fix \Psi_{\varphi, Q, P} =\bigcap_{{\tiny \begin{array}{c} \varphi \in S_0 \\ \mathcal{H} \leqslant \Fix {\Psi_{\varphi, Q, P}} \end{array}}} \Fix \Psi_{\varphi, Q, P}.
 $$
\end{prop}

\begin{proof}
The first equality is just the definition of endo-closure of $\HH$, taking into account the fact that, since $\HH$ is not abelian, $\stab_{\GG}(\HH)$ contains no type-(II) endomorphism. 

Consider the projection subgroup $\HH\pi\leqslant \FF_n$. Using  Theorem~\ref{thm for free endo}, compute $k\leq 2n$ endomorphisms $\varphi_1, \ldots, \varphi_k \in \End(\FF_n)$ such that $\ecl_{\FF_n}(\HH \pi)= \Fix \varphi_1 \cap \cdots \cap \Fix \varphi_k$; and let $S_0=\{\varphi_0, \varphi_1, \ldots, \varphi_k\}$, where $\varphi_0$ is the identity. Obviously, $\Fix S_0=\ecl_{\FF_n}(\mathcal{H}\pi)$.

It is clear that there may be endomorphisms of $\mathcal{G}$ fixing $\mathcal{H}$ but having their projections outside $S_0$. This fact leads to the inclusion  
 \begin{equation}\label{inclusion2}
\bigcap_{{\tiny \begin{array}{c} \mathcal{H} \leqslant \Fix {\Psi_{\varphi, Q, P}} \end{array}}}  \Fix \Psi_{\varphi, Q, P} \,\,\leqslant \bigcap_{{\tiny \begin{array}{c} \varphi \in S_0 \\ \mathcal{H} \leqslant \Fix {\Psi_{\varphi, Q, P}} \end{array}}} \Fix \Psi_{\varphi, Q, P}.
 \end{equation}
Now we claim that, for any $Q\in \Mat_{m\times m}(\Z)$ and $P\in \Mat_{n\times m}(\Z)$, we have $\Fix \Psi_{\varphi, Q, P}\leqslant \Fix \Psi_{\id, Q, P}$. In fact, for any $t^{\mathbf{a}}u\in \Fix \Psi_{\varphi, Q, P}$, we have ${\mathbf{a}Q+\mathbf{u}P}=\mathbf{a}$ and $u\varphi=u$, and hence $(t^{\mathbf{a}}u) \Psi_{\id, Q, P}=t^{\mathbf{a}Q+ \mathbf{u}P}u=t^{\mathbf{a}}u$ so, $t^{\mathbf{a}}u\in \Fix \Psi_{\id, Q, P}$. This proves the claim. In particular, if $\Psi_{\varphi, Q,P}\in \stab_{\GG}(\HH)$ then $\Psi_{\id, Q, P}\in \stab_{\GG}(\HH)$ as well. 

We are ready to prove the opposite inequality from Eq.~\eqref{inclusion2}. Let $t^{\mathbf{a}}u$ be an arbitrary element from the right hand side, i.e., fixed by all the endomorphisms $\Psi_{\varphi, Q, P}\in \stab_{\GG}(\HH)$ with $\Psi\pi=\varphi \in S_0$. In particular, $u\varphi=u$ for all $\varphi \in S_0$, which implies $u\in \Fix(S_0)=\ecl_{\FF_n}(\HH\pi)$ and, by Lemma~\ref{prop A}, $u\varphi=u$ for all $\varphi \in (\stab_{\GG}(\HH))\pi$. Now let $\Psi_{\varphi', Q', P'}$ be an arbitrary endomorphism in $\stab_{\GG}(\HH)$; in particular, $u\varphi'=u$. From the previous claim, $\Psi_{\id, Q', P'}\in \stab_{\GG}(\HH)$. But $(\Psi_{\id, Q', P'})\pi=\id \in S_0$ hence, $t^{\mathbf{a}}u \in \Fix\Psi_{\id, Q', P'}$ and so, $\mathbf{a}Q'+ \mathbf{u}P'=\mathbf{a}$. Altogether, we have $u\varphi'=u$ and $\mathbf{a}Q'+ \mathbf{u}P'=\mathbf{a}$ which in turn implies that $t^{\mathbf{a}}u \in \Fix \Psi_{\varphi', Q', P'}$. We deduce that $t^{\mathbf{a}}u \in \bigcap_{\tiny \,\, \HH\leqslant \Fix {\Psi_{\varphi, Q, P}}} \Fix \Psi_{\varphi, Q, P} =\ecl_{\GG}(\HH)$, concluding the proof. 
\end{proof}

\begin{prop}\label{prop B}
Let $\HH\leqfg \GG$ be a non-abelian (given by a finite set of generators). There exists finitely many computable pairs of matrices $(Q_1, P_1), \ldots, (Q_q, P_q)$ (of sizes $m\times m$ and $n\times m$, respectively) such that, for each $\varphi \in (\stab_{\mathcal{G}}(\mathcal{H}))\pi$,  
 \begin{equation}\label{eq fix}
\bigcap\limits_{{\tiny \begin{array}{c} P,Q \\ \mathcal{H} \leqslant \Fix {\Psi_{\varphi, Q, P}} \end{array}}}  \Fix \Psi_{\varphi, Q, P}= \Fix {\Psi_{\varphi, Q_1, P_1}} \cap \cdots \cap \Fix {\Psi_{\varphi, Q_q, P_q}}.
 \end{equation}
\end{prop}

\begin{proof}
Let $\mathcal{H}= \langle t^{\mathbf{a_1}}u_1, \ldots, t^{\mathbf{a_r}}u_r, t^{\mathbf{b_1}}, \ldots, t^{\mathbf{b_s}} \rangle$. Note that, for any type-(I) endomorphism $\Psi_{\varphi, Q, P}$, the condition $\HH\leqslant \Fix \Psi_{\varphi, Q, P}$ is equivalent to saying 
 \begin{equation}\label{equ homo}
\begin{array}{r} \mathbf{a_i}(Q-I)+\mathbf{u_i}P=\mathbf{0}; \phantom{aa} i=1, \ldots, r,\\ \mathbf{b_j} (Q-I)=\mathbf{0}; \phantom{aa} j=1, \ldots, s, \end{array}  
 \end{equation}
and $u_i\varphi= u_i$, for $i=1,\ldots ,r$; observe also that conditions~\eqref{equ homo} \emph{do not} depend on $\varphi$. 

Consider the system of linear equations given by 
 \begin{equation}\label{r equ homo}
\begin{array}{r} \mathbf{a_i}X+\mathbf{u_i} Y=\mathbf{0}; \phantom{aa} i=1, \ldots, r, \\ \mathbf{b_j} X=\mathbf{0}; \phantom{aa} j=1, \ldots, s, \end{array}  
 \end{equation}
where $X\in \Mat_{m\times m}(\ZZ)$ and $Y\in \Mat_{n\times m}(\ZZ)$ (this is the same as~\eqref{equ homo} after the change of variable $X=Q-I$ and $Y=P$). Equation~\eqref{r equ homo} is an homogeneous linear system of $mr+ms$ equations and $m^2+nm$ unknowns, having rank $p$, say. So, all its solutions form a subgroup of $\Mat_{m\times m}(\ZZ)\times \Mat_{n\times m}(\ZZ)$ of rank $q=m^2+nm-p$; that is, we can compute integral matrices $(X_1, Y_1), \ldots, (X_q, Y_q)$ of the corresponding sizes, such that each solution $(X,Y)$ of~\eqref{r equ homo} is of the form 
 $$
(X,Y) = \lambda_1(X_1, Y_1)+ \cdots + \lambda_q(X_q, Y_q),
 $$
for some (unique) integral parameters $\lambda_1, \ldots, \lambda_q\in \ZZ$.

Now, consider the $q$ pairs of matrices $(Q_1, P_1), \ldots ,(Q_q, P_q)$ given by $Q_i=X_i+I$ and $P_i=Y_i$, $i=i,\ldots ,q$. Fix an endomorphism $\varphi \in (\stab_{\mathcal{G}}(\mathcal{H}))\pi$, and let us prove Equation~\eqref{eq fix} from the statement. The inclusion to the right is immediate as in the left hand side we are intersecting more subgroups than in the right hand side. 

To prove the other inclusion, let $t^{\mathbf{a}}u$ be an arbitrary element from $\Fix {\Psi_{\varphi, Q_1, P_1}} \cap \cdots \cap \Fix {\Psi_{\varphi, Q_q, P_q}}$. In particular, $u\varphi =u$ and 
 \begin{equation}\label{standard equ for Q}
\mathbf{a}(Q_i-I )+\mathbf{u} P_i=\mathbf{0}, \phantom{aa} i=1, \ldots, q.
 \end{equation}
Let $\Psi_{\varphi, Q', P'}$ be an arbitrary endomorphism of $\mathcal{G}$ such that $\mathcal{H} \leqslant \Fix \Psi_{\varphi, Q', P'}$. Then $(Q', P')$ satisfies~\eqref{equ homo} and $(X', Y')=(Q'-I, P')$ is a solution of~\eqref{r equ homo} so, there exist $\lambda'_1, \ldots, \lambda'_q$ such that $(Q'-I, P')=\lambda'_1(Q_1-I, P_1) +\cdots +\lambda'_q(Q_q-I, P_q)$. From~\eqref{standard equ for Q} we have $\sum^q_{i=1}\lambda'_i(\mathbf{a}(Q_i-I)+\mathbf{u} P_i)=\mathbf{0}$ and so, $\mathbf{a}(\sum^q_{i=1}\lambda'_i(Q_i-I))+\mathbf{u} \sum^q_{i=1}\lambda'_i P_i =\mathbf{0}$, i.e., $\mathbf{a}(Q'-I)+ \mathbf{u} P'=\mathbf{0}$. Therefore, $t^{\mathbf{a}}u \in \Fix \Psi_{\varphi, Q', P'}$. This shows that $t^{\mathbf{a}}u$ belongs to the left hand side of~\eqref{eq fix} and the proof is complete. 
\end{proof}

Finally, we can prove the main result. 

\begin{proof}[Proof of Theorem~\ref{main}]
Let $\mathcal{H}= \langle t^{\mathbf{a_1}}u_1, \ldots, t^{\mathbf{a_r}}u_r, t^{\mathbf{b_1}}, \ldots, t^{\mathbf{b_s}}\rangle$. By Proposition~\ref{prop main}, we can compute a set of endomorphisms $S_0=\{ \id =\varphi_0, \varphi_1, \ldots ,\varphi_k\}\subseteq \stab_{\FF_n}(\HH\pi)$, $k\leqslant 2n$, such that
 $$
\ecl_{\mathcal{G}}(\mathcal{H})= \bigcap_{{\tiny \begin{array}{c} \mathcal{H} \leqslant \Fix {\Psi_{\varphi, Q, P}} \end{array}}} \Fix \Psi_{\varphi, Q, P} =\bigcap_{{\tiny \begin{array}{c} \varphi \in S_0 \\ \mathcal{H} \leqslant \Fix {\Psi_{\varphi, Q, P}} \end{array}}} \Fix \Psi_{\varphi, Q, P}.
 $$
On the other hand, by Proposition~\ref{prop B}, we can compute a set of pairs of integral matrices $\{ (Q_1, P_1), \ldots, (Q_q, P_q)\}$ (of sizes $m\times m$ and $n\times m$, respectively) such that, for each $\varphi \in (\stab_{\mathcal{G}}(\mathcal{H}))\pi$,  
 $$
\bigcap\limits_{{\tiny \begin{array}{c} P,Q \\ \mathcal{H} \leqslant \Fix {\Psi_{\varphi, Q, P}} \end{array}}}  \Fix \Psi_{\varphi, Q, P}= \Fix {\Psi_{\varphi, Q_1, P_1}} \cap \cdots \cap \Fix {\Psi_{\varphi, Q_q, P_q}}.
 $$
Putting all together, we get 
 $$
\ecl_{\mathcal{G}}(\mathcal{H}) =\bigcap_{{\tiny \begin{array}{c} \varphi \in S_0 \\ \mathcal{H} \leqslant \Fix {\Psi_{\varphi, Q, P}} \end{array}}} \Fix \Psi_{\varphi, Q, P}=\bigcap\limits_{i=0}^{k} \bigcap\limits_{{\tiny \begin{array}{c} P,Q \\ \mathcal{H} \leqslant \Fix {\Psi_{\varphi_i, Q, P}} \end{array}}}  \Fix \Psi_{\varphi_i, Q, P}=\bigcap\limits_{i=0}^{k} \,\,\, \bigcap\limits_{j=1}^{q} \Fix \Psi_{\varphi_i, Q_j, P_j}.
 $$
This shows the first part of the result, with $\mathcal{S}_0=\{\Psi_{\varphi_i, Q_j, P_j} \mid i=0,\ldots ,k,\,\, j=1,\ldots ,q\}$ (and  $\ell=(k+1)q$). 

Finally, apply the endomorphism version of Theorem~\ref{RV-fix} (see Remark~\ref{remark}) to decide whether $\Fix \mathcal{S}_0=\ecl_{\GG}(\HH)$ is finitely generated and, in case it is, compute a basis for it. Note additionally that, just for the purpose of the present proof we do not need to use Mutanguha's result: the computability of a basis of $\Fix \varphi_0 \cap\Fix \varphi_1 \cap \cdots \cap \Fix \varphi_k =\ecl_{\FF_n}(\HH\pi)$ can be done directly by invoking Theorem~\ref{thm for free endo}, instead of computing bases for each of the individual fixed subgroups and then intersecting them all. 
\end{proof}


\begin{proof}[Proof of Corollary~\ref{maincor}]
The proof is straightforward: apply Theorem~\ref{main} to $\HH$; if $\ecl_{\GG}(\HH)$ is strictly bigger than $\HH$, then $\HH$ is not endo-fixed (there are elements outside $\HH$ which are fixed by \emph{every} endomorphism of $\GG$ fixing $\HH$). Otherwise, $\ecl_{\GG}(\HH)=\HH$ and the algorithm in Theorem~\ref{main} also outputs a list of $\ell=(k+1)q$ endomorphisms $\Psi_1, \ldots ,\Psi_{\ell} \in \End(\GG)$, such that $\Fix\Psi_1 \cap \cdots \cap \Fix\Psi_\ell =\ecl_{\GG}(\HH)=\HH$.
\end{proof}

\section{Open questions}\label{open}

We end up asking a few related questions:

\begin{que}
Let $\GG=\Z^m\times \FF_n$. Is there an algorithm which, given a finite set of generators for a subgroup $\HH\leqfg \GG$, decides whether the monoid $\stab_{\GG}(\HH)$ is finitely generated or not and, in case it is, computes a set of endomorphisms $\Psi_1, \ldots, \Psi_\ell \in \End(\GG)$ such that $\stab_{\GG}(\HH)=\langle \Psi_1, \ldots, \Psi_\ell \rangle$~?
\end{que}

\begin{que}
Is it true that, for every $\HH\leqfg \GG=\Z^m\times \FF_n$, the auto-fixed closure $\acl_{\GG}(\HH)$ is again finitely generated ? What about the endo-fixed closure $\ecl_{\GG}(\HH)$~?
\end{que}

\begin{que}
What is the complexity of the proposed algorithms~? Can we do more efficient ones, say, polynomial~?
\end{que}

\section*{Acknowledgements}

\noindent The three authors acknowledge support from the Spanish Agencia Estatal de Investigaci\'on through grant PID2021-126851NB-100 (AEI/ FEDER, UE). The first named author wants to thank the hospitality received from Universidad del Pa\'{\i}s Vasco, and support through a Margarita Salas grant from Universitat Polit\`ecnica de Catalunya.


\end{document}